\documentclass[12pt]{article}
\usepackage{latexsym,amsfonts,amssymb,mathrsfs,float}
\usepackage{dsfont}
\usepackage{amsthm}
\usepackage[numbers]{natbib}
\usepackage{amsmath}
\usepackage{indentfirst}
\usepackage{enumerate}
\usepackage{color}

\allowdisplaybreaks

\newtheorem{thm}{Theorem}[section]

\newtheorem{lem}[thm]{Lemma}

\newtheorem{rem}[thm]{Remark}

\numberwithin{equation}{section}

\begin{document}

\title{\bf The Asymptotic Behavior of Rarely Visited Edges of the Simple  Random Walk}

\author{Ze-Chun Hu$^a$, Xue Peng$^{a,}$\thanks{Corresponding author}\ , Renming Song$^b$, Yuan Tan$^a$\\ \\ \\
 {\small $^a$ College of Mathematics, Sichuan  University, Chengdu, 610065,  China}\\
 {\small zchu@scu.edu.cn, pengxuemath@scu.edu.cn, 630769091@qq.com}\\ \\
 {\small $^b$ Department of Mathematics, University of Illinois Urbana-Champaign, Urbana, IL 61801, USA}\\
{\small rsong@illinois.edu}}

\maketitle

\begin{abstract}
 In this paper, we study the asymptotic behavior of the number of rarely visited edges (i.e., edges that are visited only once) of
 a simple symmetric random walk on $\mathbb{Z}$. Let $\alpha(n)$ be the number of
 rarely visited edges up to time $n$.
 First we evaluate $\mathbb{E}(\alpha(n))$, show that $n\to  \mathbb{E}(\alpha(n))$ is non-decreasing in $n$ and that $\lim\limits_{n\to \infty}\mathbb{E}(\alpha(n))=2$. Then we study the asymptotic behavior of $\mathbb{P} (\alpha(n)>a(\log n)^2)$
 for any $a>0$ and use it to show that there exists a constant $C\in (\frac{1}{32}, \frac{1}{2}]$ such that $\limsup\limits_{n\to \infty}\frac{\alpha(n)}{(\log n)^2}=C$ almost surely.
 \end{abstract}

\smallskip


\smallskip

\noindent {\bf 2020 MR Subject Classification (2020)}\quad 60F15; 60J55

\section{Introduction and the main results}

Suppose that  $\{S_n\}_{n\geq 0}$ is a simple  symmetric random walk
on $\mathbb{Z}$ with $S_0=0$, defined on the probability space $(\Omega,\mathscr{F},\mathbb{P})$.
Let $X_n:=S_n-S_{n-1}, n\geq 1$. Then $\{X_n, n\ge 1\}$ are \ i.i.d. with $\mathbb{P}(X_1=1)=\mathbb{P}(X_1=-1)=\frac{1}{2}$.

For $y\in\mathbb{Z}$, we use $\xi(y,n):=\#\{0\leq k\leq n: S_k=y\}$ to denote
the time spent at $y$ by  $\{S_m\}_{m\geq 0}$ up to time $n$.
Here and throughout this paper, $\# D$ denotes the cardinality of the set $D$.
A site $x\in\mathbb{Z}$ is called a {\it favorite (most visited) site} of  $\{S_m\}_{m\geq 0}$ up to time $n$ if
$$
\xi(x,n)=\max_{y\in{\mathbb{Z}}}\xi(y,n).
$$

For $y\in\mathbb{Z}$, we use $\langle y, y+1\rangle$ to denote the edge between the sites $y$ and $y+1$. The numbers of upcrossings and downcrossings  of $y\in \mathbb{Z}$ by $\{S_m\}_{m\geq 0}$ up to time $n\geq 1$ are defined by
\begin{align*}
L^U(y,n):=\#\{0<k\leq n: S_k=y, S_{k-1}=y-1\},\\
L^D(y,n):=\#\{0<k\leq n: S_k=y, S_{k-1}=y+1\}.
\end{align*}
Set
$$
L(y,n):=L^U(y+1,n)+L^D(y,n).
$$
Then $L(y,n)$ is the number of times that  $\{S_m\}_{m\geq 0}$ visits the edge $\langle y, y+1\rangle$ up to time $n$.
An edge $\langle x, x+1\rangle $ is called a {\it favorite edge} of $\{S_m\}_{m\geq 0}$ up to time $n$ if
$$
L(x,n)=\sup_{y\in{\mathbb{Z}}}L(y,n).
$$

The study of favorite sites of
random walks was initiated by Erd\"{o}s and R\'{e}v\'{e}sz \cite{ER84}. Since then, this
topic  has been intensively studied,  see Bass \cite{B22},
Bass and Griffin \cite{BG85},
Ding and Shen \cite{DS18}, Erd\"{o}s and R\'{e}v\'{e}sz \cite{ER87, ER91},  Hao \cite{H23-b}, Hao et al. \cite{HHMS22, HHMS23},  Shi and T\'{o}th \cite{ST00}, T\'{o}th \cite{To01}, T\'{o}th and Werner \cite{TW97} and the references therein.

A site  $x\in\mathbb{Z}$
 is called a {\it rarely visited site} of  $\{S_m\}_{m\geq 0}$ up to time $n$ if $\xi(x,n)=1$.
Compared to favorite sites,  there are only a few papers on rarely visited
sites, see Major \cite{Ma88}, Newman \cite{Ne84}  and T\'{o}th \cite{To96}.
 Following R\'{e}v\'{e}sz \cite{Re90}, we use $f_1(n)$ to denote the number of
 rarely visited sites up to time $n$, i.e.,
\begin{align}\label{e:raresites}
f_1(n):=\#\{x\in\mathbb{Z}: \xi(x,n)=1\}.
\end{align}
Newman \cite{Ne84} proved that $\mathbb{E}(f_1(n))=2$, for all $n\geq 1$. Major \cite{Ma88} proved that there exists a constant $C\in (0,\infty)$ such that $\limsup\limits_{n\to\infty}\frac{f_1(n)}{(\log n)^2}=C$ almost surely.

An edge $\langle x, x+1\rangle$ is called a {\it rarely visited edge}  of  $\{S_m\}_{m\geq 0}$ up to time $n$ if $L(x,n)=1$.
So far it seems that no one has studied rarely visited edges. The purpose of this paper is to study the  asymptotic behavior of the number of rarely visited edges.
Define
\begin{equation}\label{eq: defn edge}
\mathcal{A}_n:=\{\langle x,x+1\rangle
: L(x,n)=1\},\quad \alpha(n):=\#\mathcal{A}_n, \,\,n\geq 1.
\end{equation}
Then $\mathcal{A}_n$ is the collection of all the rarely visited edges of $\{S_m\}_{m\geq 0}$ up to time $n$, and $\alpha(n)$ is the number of  rarely visited edges of  $\{S_m\}_{m\geq 0}$ up to time $n$.
The main results of this paper are as follows:

\begin{thm}\label{thm: expectation of alpha(n)}
	(i) $\mathbb{E}(\alpha(1))=1$ and for all $n\geq 1$,
	\begin{equation}
		\left\{\begin{array}{llll}
			\mathbb{E}(\alpha(n+1)) & = & \mathbb{E}(\alpha(n)), & \text{if}\ n\ \text{is odd},\\
			\mathbb{E}(\alpha(n+1)) & = & \mathbb{E}(\alpha(n))+2\cdot \frac{(n-1)!!}{(n+2)!!}, & \text{if}\  n\ \text{is even}.
		\end{array}\right.
	\end{equation}
(ii) $\lim\limits_{n\rightarrow \infty}\mathbb{E}(\alpha(n))=2$.
\end{thm}

\begin{thm}\label{thm: inequality of alpha(n)}
For all $a>0$ and $\varepsilon>0$, there exists an $N_0=N_0(a,\varepsilon)$ such that for all $n>N_0$,
$$
n^{-2a-\varepsilon  }<\mathbb{P}\left(\alpha(n)>a(\log n)^2\right)<n^{-2a+\varepsilon}.
$$
\end{thm}

\begin{thm}\label{thm: limitation of alpha(n)}
There exists a constant
$C\in (\frac{1}{32}, \frac{1}{2}]$ such that
$$
\mathbb{P}\left(\limsup\limits_{n\rightarrow\infty}\frac{\alpha(n)}{(\log n)^2}=C\right)=1.
$$
\end{thm}	

From Theorem \ref{thm: expectation of alpha(n)}, we can see that, unlike the result that
the expected number $\mathbb{E}(f_1(n))$ of rarely visited sites is equal to 2 for all $n\geq 1$,
the expected number of rarely visited edges $\mathbb{E}(\alpha(n))$ increases
with $n$ and $\lim_{n\to\infty}\mathbb{E}(\alpha(n))=2$.
Theorem \ref{thm: inequality of alpha(n)} and Theorem \ref{thm: limitation of alpha(n)}
imply that the asymptotic behavior of rarely visited edges is similar to that of  rarely visited sites.

\begin{rem} Related to the results above, we think the following
problems are worth studying in the future:
\begin{enumerate}[(1)]
\item What is the exact value of the constant $C$ in Theorem \ref{thm: limitation of alpha(n)}?

\item
Is  the value of the constant $C$ in Theorem \ref{thm: limitation of alpha(n)}
the same as that of the corresponding constant in the case of rarely visited sites?

\end{enumerate}
\end{rem}

The rest of the paper is organized as follows. In Section 2, we give the proof of Theorem \ref{thm: expectation of alpha(n)}. In Section 3, the proofs of Theorems \ref{thm: inequality of alpha(n)} and \ref{thm: limitation of alpha(n)} will be given.

\section{Proof of Theorem \ref{thm: expectation of alpha(n)}}\setcounter{equation}{0}

Our proof of Theorems \ref{thm: expectation of alpha(n)} is inspired by Newman \cite{Ne84}. We spell out the details for the reader's convenience.
We will point out the main modifications in Remark \ref{rem-2.1} below.

Without loss of generality, for the proof of Theorem \ref{thm: expectation of alpha(n)}, we can assume that
$$
\Omega:=\{\omega=(\omega_0, \omega_1,\omega_2,\dots): \omega_n\in\mathbb{Z},\,\, \forall n\ge0, \mbox{ and } |\omega_n-\omega_{n-1}|=1,\,\, \forall n\geq 1 \}.
$$
Let $\mathscr{F}$ be the $\sigma$-field on $\Omega$ generated by
all the cylinder sets.
For $n\ge 0$, $x\in \mathbb{Z}$ and $x_0, x_1, \dots, x_n\in \mathbb{Z}$ satisfying $|x_k-x_{k-1}|=1$ for all $k=1, \dots, n$, we define a probability measure $\mathbb{P}_x$ on $(\Omega, \mathscr{F})$ by
$$
\mathbb{P}_x(\omega: \omega_0=x_0,\omega_1=x_1,,\dots,\omega_n=x_n):=\frac{1}{2^n}\delta_x(x_0).
$$
Let
$$
S_n(\omega):=\omega_n, \,\, \forall n\ge 0, \quad X_0:=S_0, \,\,X_n=S_n-S_{n-1}, \,\, \forall n\ge 1.
$$
Then under $\mathbb{P}_x$, $\{S_n\}_{n\ge 0}$ is a simple symmetric random walk on $\mathbb{Z}$ with $S_0=x$,
and $\{X_n\}_{n\ge 1}$ are
i.i.d. random variables with
$$
\mathbb{P}_x(X_1=1)=\mathbb{P}_x(X_1=-1)=\frac{1}{2}.
$$
$\mathbb{P}_0$ is the probability measure $\mathbb{P}$ of Section 1.
We will use $\mathbb{E}_x$ to denote
the expectation with respect to $\mathbb{P}_x$.

\begin{proof}[\bf Proof of Theorem \ref{thm: expectation of alpha(n)}]
 (i) Obviously, we have $\mathbb{E}_0(\alpha(1))=1$.

 Let $\tilde{\alpha}(n)$ be the number of
 rarely visited edges of the random walk $\{S_k, 1\leq k\leq n+1\}$.
 Since $X_1,X_2,\dots,X_{n+1}$ are i.i.d., we have
\begin{align}\label{proof-thm1.1-a}
		\mathbb{E}_0(\tilde{\alpha}(n))&=\mathbb{P}_0(X_1=1)\mathbb{E}_0(\tilde{\alpha}(n)|X_1=1)
+\mathbb{P}_0(X_1=-1)\mathbb{E}_0(\tilde{\alpha}(n)|X_1=-1)\notag\nonumber\\
&=\frac{1}{2}\mathbb{E}_1(\alpha(n))+\frac{1}{2}\mathbb{E}_{-1}(\alpha(n))\nonumber\\
&=\frac{1}{2}\mathbb{E}_0(\alpha(n))+\frac{1}{2}\mathbb{E}_{0}(\alpha(n))\nonumber\\
&=\mathbb{E}_0(\alpha(n)).
	\end{align}
By the definitions of $\alpha(n+1)$ and $\tilde{\alpha}(n)$, we have
\begin{align*}
\mathbb{E}_0(\alpha(n+1))&=\mathbb{E}_0(\alpha(n+1);X_1=1)+\mathbb{E}_0(\alpha(n+1);X_1=-1)\nonumber\\
&=[\mathbb{E}_0(\alpha(n+1);L(0,n+1)=1,X_1=1)
+\mathbb{E}_0(\alpha(n+1);L(0,n+1)=2,X_1=1)\nonumber\\
&\quad\quad+\mathbb{E}_0(\alpha(n+1);L(0,n+1)\geq 3,X_1=1)]\nonumber\\
&\quad+[\mathbb{E}_0(\alpha(n+1);L(-1,n+1)=1,X_1=-1)
+\mathbb{E}_0(\alpha(n+1);L(-1,n+1)=2,X_1=-1)\nonumber\\
&\quad\quad +\mathbb{E}_0(\alpha(n+1);L(-1,n+1)\geq 3,X_1=-1)]\nonumber\\
&=[\mathbb{E}_0(\tilde{\alpha}(n)+1;L(0,n+1)=1,X_1=1)
+\mathbb{E}_0(\tilde{\alpha}(n)-1;L(0,n+1)=2,X_1=1)\nonumber\\
&\quad\quad+\mathbb{E}_0(\tilde{\alpha}(n);L(0,n+1)\geq 3,X_1=1)]\nonumber\\
&\quad+[\mathbb{E}_0(\tilde{\alpha}(n)+1;L(-1,n+1)=1,X_1=-1)
+\mathbb{E}_0(\tilde{\alpha}(n)-1;L(-1,n+1)=2,X_1=-1)\nonumber\\
&\quad\quad +\mathbb{E}_0(\tilde{\alpha}(n);L(-1,n+1)\geq 3,X_1=-1)]\nonumber\\
&=\mathbb{E}_0(\tilde{\alpha}(n))+[\mathbb{P}_0(L(0,n+1)=1, X_1=1)-\mathbb{P}_0(L(0,n+1)=2,X_1=1)]\notag\\
&\quad+[\mathbb{P}_0(L(-1,n+1)=1, X_1=-1)-\mathbb{P}_0(L(-1,n+1)=2,X_1=-1)],
\end{align*}
which together with (\ref{proof-thm1.1-a}) implies that
\begin{align}\label{proof-thm1.1-b}
&\mathbb{E}_0(\alpha(n+1)-\alpha(n))\nonumber\\
&=[\mathbb{P}_0(L(0,n+1)=1, X_1=1)-\mathbb{P}_0(L(0,n+1)=2,X_1=1)]\notag\\
&\quad+[\mathbb{P}_0(L(-1,n+1)=1, X_1=-1)-\mathbb{P}_0(L(-1,n+1)=2,X_1=-1)].
\end{align}
	
For $\omega\in\Omega$, define
	\begin{equation}\label{eq: sup time}
		\sigma(\omega):=\sup\{0<k\leq n+1, S_{k}(\omega)=1\}.
	\end{equation}
	Then
	\begin{align}\label{eq: sigma and n+1}
		\mathbb{P}_0(L(0,n+1)=1, X_1=1)
		&=\mathbb{P}_0(L(0,n+1)=1, X_1=1,\sigma<n+1)\notag\\
		&\quad+\mathbb{P}_0(L(0,n+1)=1, X_1=1,\sigma=n+1).
	\end{align}

We deal with $\mathbb{P}_0(L(0,n+1)=1, X_1=1,\sigma<n+1)$ first.
For any $\omega\in\{L(0,n+1)=1, X_1=1,\sigma\le n+1\}$,  define $\omega'\in \Omega$ by
	\begin{equation}\label{eq: construct path01}
\omega'_k:=\left\{
\begin{array}{cl}
			\omega_k, &\mbox{if}\quad 0\leq k\leq\sigma(\omega),\\
			2-\omega_k, & \mbox{if}\quad k> \sigma(\omega).
		\end{array}\right.
	\end{equation}
One can check that  the map defined by (\ref{eq: construct path01}) is a one-to-one correspondence between the sets $\{L(0,n+1)=1,S_0=0,X_1=1,\sigma<n+1\}$ and $\{L(0,n+1)=2,S_0=0,X_1=1\}$.
It follows that
	\begin{equation}\label{eq: n prob 01}
		\mathbb{P}_0(L(0,n+1)=1,X_1=1, \sigma<n+1)=\mathbb{P}_0(L(0,n+1)=2,X_1=1).
	\end{equation}

Now we deal with $\mathbb{P}_0(L(0,n+1)=1, X_1=1, \sigma=n+1)$. Note that
	\begin{align}\label{eq: n prob 02}
		&\mathbb{P}_0(L(0,n+1)=1, X_1=1, \sigma=n+1)\notag\\
		&=\mathbb{P}_0(L(0,n+1)=1, X_1=1, S_n=2, S_{n+1}=1)\notag\\
		&=\mathbb{P}_0(X_1=1)\mathbb{P}_0(S_j\geq 1,\,\,1\leq j\leq n+1, S_{n+1}=1|X_1=1)\notag\\
		&=\frac{1}{2}\mathbb{P}_0(S_j\geq 0, \,\,0\leq j\leq n, S_n=0).
	\end{align}

Combining  \eqref{eq: sigma and n+1}, \eqref{eq: n prob 01} and \eqref{eq: n prob 02}, we get
\begin{align}\label{eq:  prob equal 03}
		&\mathbb{P}_0(L(0,n+1)=1,X_1=1)\notag\\
		&=\mathbb{P}_0(L(0,n+1)=2,X_1=1)+\frac{1}{2}\mathbb{P}_0(S_j\geq 0, 0\leq j\leq n, S_n=0).
\end{align}
Similarly, by the symmetry of $\{S_m\}_{m\geq 0}$, we have
\begin{align}\label{eq:  prob 04}
&\mathbb{P}_0(L(-1,n+1)=1,X_1=-1)\notag\\
&=\mathbb{P}_0(L(-1,n+1)=2,X_1=-1)+\frac{1}{2}\mathbb{P}_0(S_j\leq 0, 0\leq j\leq n, S_n=0)\notag\\
&=\mathbb{P}_0(L(-1,n+1)=2,X_1=-1)+\frac{1}{2}\mathbb{P}_0(S_j\geq 0, 0\leq j\leq n, S_n=0).
\end{align}
By \eqref{proof-thm1.1-b}, \eqref{eq:  prob equal 03} and \eqref{eq:  prob 04}, we obtain
\begin{align}\label{eq: expectation n+1 minus n}
\mathbb{E}_0(\alpha(n+1))-\mathbb{E}_0(\alpha(n))=\mathbb{P}_0(S_j\geq 0, 0\leq j\leq n, S_n=0).
\end{align}

When $n$ is odd, we have
	\begin{equation}\label{eq: n prob 03 odd}
		\mathbb{P}_0(S_j\geq 0, \,\,0\leq j\leq n, S_n=0)=0.
	\end{equation}

When $n$ is even, we can express the probability above in terms of the Catalan numbers. Recall that the $k$-th Catalan number is defined by
$$
C_k:=\frac{1}{k+1}\binom{2k}{k},\ k\geq 0,
$$
which stands for the number of Dyck paths at time $2k$. A Dyck path  at time $2k$ is a random walk of $2k$ steps that begins at 0, ends at 0 ($k$ up steps, $k$ down steps) and never goes below 0 (nonnegative).
Then we get that when $n=2k$,
\begin{align}\label{eq: n prob 03}
\mathbb{P}_0(S_j\geq 0, \,\,0\leq j\leq n, S_n=0)
&=\frac{C_k}{2^n}=\frac{1}{2^n}\cdot\frac{1}{k+1}\binom{2k}{k}=2\cdot \frac{(n-1)!!}{(n+2)!!}.
\end{align}

By \eqref{eq: expectation n+1 minus n}, \eqref{eq: n prob 03 odd} and \eqref{eq: n prob 03}, we obtain
\begin{align*}
\mathbb{E}(\alpha(n+1))-\mathbb{E}(\alpha(n))
&=\mathbb{E}_0(\alpha(n+1))-\mathbb{E}_0(\alpha(n))\notag\\
&=\mathbb{P}_0(S_j\geq 0, 0\leq j\leq n, S_n=0)\notag\\
&=\left\{
\begin{array}{cl}
0,&\mbox{if}\ n\ \text{is odd},\\
2\cdot \frac{(n-1)!!}{(n+2)!!},&\mbox{if}\ n\ \text{is even}.
\end{array}
\right.
\end{align*}
	
(ii)
Recall the Taylor expansion
\begin{align*}
\sqrt{1-x}=1-\frac{1}{2}x-\sum_{m=2}^{\infty}\frac{(2m-3)!!}{(2m)!!}
x^m, \qquad x\in[-1,1].
\end{align*}
This, together with (i), implies that
\begin{align*}
\lim\limits_{n\rightarrow\infty}\mathbb{E}(\alpha(n+1))
&=\mathbb{E}(\alpha(1))+\lim\limits_{n\rightarrow\infty}\sum^n_{k=1}
[\mathbb{E}(\alpha(k+1))-\mathbb{E}(\alpha(k))]\\
&=1+2\sum^{\infty}_{k=1}\frac{(2k-1)!!}{(2k+2)!!}\\
&=1+2\left(1-\frac{1}{2}\right)=2.
\end{align*}
The proof is complete.
\end{proof}

\smallskip

\begin{rem}\label{rem-2.1}
The basic idea of the above proof comes from Newman \cite{Ne84}.
But the case of rarely visited edges is more complicated to deal with than the case of rarely visited sites.
In the case of rarely visited site sites in  Newman \cite{Ne84}, it
is relatively easy to show that
$E_0(f_1(n+1))=E_0(f_1(n))=E_0(f_1(1))\equiv2$.
However, in the case
rarely visited edges, $E_0(\alpha(n))$ is not constant and things are more complicated.
\end{rem}

\section{Proofs of Theorems \ref{thm: inequality of alpha(n)} and \ref{thm: limitation of alpha(n)}}\setcounter{equation}{0}

Our proofs of Theorems \ref{thm: inequality of alpha(n)} and  \ref{thm: limitation of alpha(n)} are inspired by Major \cite{Ma88}.

\subsection{Some preparations}


It follows from \cite[Lemma 3 and Remark 6]{Ma88} that
\begin{align}\label{eq: L}
&\lim_{n\to\infty}n\mathbb{P}\left(S_j>0\ \mbox{for all}\ 0<j\leq n\ \mbox{and}\ S_j<S_n\ \mbox{for all}\ 0\leq j<n\right)\notag\\
&=\lim_{n\to\infty}n\mathbb{P}(0<S_j<S_n, \,\,\text{for all}\,\, 0<j<n)=
\frac{1}{4}.
\end{align}
It is well known (see, for instance, \cite[Lemma 4.9.3]{Durrett}) that
$$
\mathbb{P}(S_1\neq 0, \dots, S_{2n}\neq 0)=\mathbb{P}(S_{2n}=0).
$$
By symmetry, we have
\begin{align}\label{e:rs2}
\mathbb{P}(S_1> 0, \dots, S_{2n}> 0)=\frac12\mathbb{P}(S_{2n}=0).
\end{align}
By Feller \cite[(3.4) of Chapter III]{Feller68}, we have
\begin{align}\label{e:rs3}
\mathbb{P}(S_1\ge 0, \dots, S_{2n}\ge 0)=\mathbb{P}(S_{2n}=0).
\end{align}

For $k\geq 0$, define
\begin{eqnarray*}
\alpha_k(n):=\left\{
\begin{array}{cl}
\binom{\alpha(n)}{k},& \mbox{if} \  k\le \alpha(n),\\
0,& \mbox{otherwise}.
\end{array}
\right.
\end{eqnarray*}
$\alpha_k(n)$ is the number of subsets of $\mathcal{A}_n $ with exactly $k$ elements.

The following lemma plays a key role in the proof of Theorem \ref{thm: inequality of alpha(n)}.

\begin{lem}\label{lem: main lem}
Let  $a>0$.
If $k\sim a \log n$ as $n\rightarrow\infty$, then for any $\varepsilon\in (0, 1/2)$,
there exists  $n_0=n_0(a,\varepsilon)$  such that for all $n\geq n_0$,
 $$
 [(\frac12-\varepsilon)\log n]^{k} <\mathbb{E}\,\alpha_k(n)<[(\frac12+\varepsilon)\log n]^{k}.
 $$
\end{lem}

For non-negative integers $r$ and $t$, we define
 \begin{align*}
 C_1(t)&:=\{0<S_l<S_t\ \text{for all}\,\, 0<l<t\}, \,\,0<t\leq n;\\
 C_2(r,t)&:=\{S_r<S_l\leq S_t\ \text{for all}\,\, r<l\leq t\}, \,\,0\leq r<t\leq n;\\
 C_2(t)&:=\{0<S_l\leq S_t\ \text{for all}\,\, 0<l\leq t\}, \,\,0<t\leq n;\\
 D_1(t)&:=\{S_l\leq S_t\ \text{for all}\,\, 0\leq l\leq t\}, \,\,0\leq t\leq n;\\
 D_2(r,t)&:=\{S_r<S_l\ \text{for all}\,\, r<l\leq t\}, \,\,0\leq r<t\leq n.
 \end{align*}

\begin{lem}\label{prop: events limit}
  (i)  $\lim\limits_{t\rightarrow\infty}t\mathbb{P}(C_2(t))=\frac12$;

  (ii) $\lim\limits_{t\rightarrow\infty}\sqrt{t}\mathbb{P}(D_1(t))=\sqrt{2/\pi}$;

  (iii) $\lim\limits_{t-r\rightarrow\infty}\sqrt{t-r}\mathbb{P}(D_2(r,t))=1/\sqrt{2\pi}$.
\end{lem}

\begin{proof}[\bf Proof.] (i) Since
\begin{align*}
\{X_{t+1}=1\}\cap C_2(t)&=\{X_{t+1}=1, 0<S_l\leq S_t\ \text{for all}\,\, 0<l\leq t\}\\
&=\{0<S_l\leq S_t<S_{t+1},\text{for all}\,\, 0<l\leq t\}\\
&=\{0<S_l<S_{t+1}\ \mbox{for all}\ 0<l<t+1\}\\
&=C_1(t+1),
\end{align*}
we have
$$
\mathbb{P}(C_1(t+1))=\mathbb{P}(\{X_{t+1}=1\}\cap C_2(t))=\mathbb{P}(X_{t+1}=1)\mathbb{P}(C_2(t))=\frac{1}{2}\mathbb{P}(C_2(t)).
$$
Thus  by  \eqref{eq: L}, we have
$$
\lim\limits_{t\rightarrow\infty}t\mathbb{P}(C_2(t))
=\lim\limits_{t\rightarrow\infty}2t\mathbb{P}(C_1(t+1))=\frac12.
$$

(ii) Let $\check{S}^t_k:=S_t-S_{t-k}$, $k=0,1,\dots,t$. Then $\{\check{S}^t_k\}_{0\leq k\leq t}$ is a simple  symmetric random walk with $\check{S}^t_0=0$. Thus by \eqref{e:rs3}, we have
$$
\mathbb{P}(D_1(t))=\mathbb{P}\left(\check{S}^t_0\geq 0, \check{S}^t_1\geq 0, \dots, \check{S}^t_t\geq 0\right)=\mathbb{P}(S_{2m}=0)=\binom{2m}{m}\frac{1}{2^{2m}},
$$
where $m=t/2$ or $m=(t+1)/2$. Then, by Stirling's formula, we have
$$
\lim_{t\rightarrow\infty}\sqrt{t}\mathbb{P}(D_1(t))=\sqrt{2/\pi}.
$$

(iii) Let $\widehat{S}^r_k=S_{r+k}-S_r$, $k=0,1,2,\dots,t-r$. Then $\{\widehat{S}^r_k\}_{0\leq k\leq t-r}$ is a simple symmetric random walk with $\widehat{S}^r_0=0$. Thus by \eqref{e:rs2}, we have
$$
\mathbb{P}(D_2(r,t))=\mathbb{P}(\widehat{S}_1>0, \widehat{S}_2>0, \dots, \widehat{S}_{t-r}>0)=\frac{1}{2}\mathbb{P}(S_{2m}=0)=\frac{1}{2}\cdot \binom{2m}{m}\frac{1}{2^{2m}},
$$
where $m=(t-r)/2$ or $m=(t-r-1)/2$. Hence, by Stirling's formula, we have
$$
\lim\limits_{t-r\rightarrow\infty}\sqrt{t-r}\mathbb{P}(D_2(r,t))=1/\sqrt{2\pi}.
$$
\end{proof}

Recall that $\mathcal{A}_n$ is defined in \eqref{eq: defn edge}.
Define
\begin{equation}\label{eq: defn positive and negative edge}
\quad \mathcal{A}^+_n:=\{z\geq 0
: \langle z,z+1\rangle\in\mathcal{A}_n\},
\quad \mathcal{A}^-_n:=\{z\leq 0
: \langle z-1,z\rangle \in\mathcal{A}_n\}.
\end{equation}
Then there is a one-to-one correspondence between $\mathcal{A}^+_n$ and the collection of rarely visited edges on the positive half-axis of  $\{S_m\}_{m\geq 0}$ up to time $n$. There is also a one-to-one correspondence between $\mathcal{A}^-_n$ and the collection of rarely visited edges on the negative half-axis of  $\{S_m\}_{m\geq 0}$ up to time $n$.
 Let $\alpha^+(n):=\#\mathcal{A}^+_n, \alpha^-(n):=\#\mathcal{A}^-_n$.

For $k\geq 0$, define
\begin{eqnarray*}
\alpha^+_k(n):=\left\{
\begin{array}{cl}
\binom{\alpha^+(n)}{k},& \mbox{if} \  k\le \alpha^+(n),\\
0,& \mbox{otherwise}.
\end{array}
\right.
\end{eqnarray*}
$\alpha^+_k(n)$ is the number of subsets of $\mathcal{A}^+_n $ with exactly $k$ elements.
$\alpha^-_k(n)$ is defined similarly.

\medskip

\begin{proof}[\bf Proof of Lemma \ref{lem: main lem}.] For $k\geq 2$,  it holds that
\begin{equation*}
	 \alpha^+_k(n)\mathbf{1}_{\{k\le \alpha^+(n)\}}
		=\sum_{0\leq j_1<\dots<j_k\leq n-1}\mathbf{1}_{D_1(j_1)C_2(j_1,j_2)C_2(j_2,j_3)\cdots C_2(j_{k-1},j_k)D_2(j_k,n)},
\end{equation*}
where $\mathbf{1}_A(\,\cdot\,)$ is the indicator function. Hence,
  \begin{align}\label{eq: alpha k general01}
  \mathbb{E}\alpha^+_k(n)&=\sum_{0\leq j_1<\dots<j_k\leq n-1}\mathbb{P}(D_1(j_1)C_2(j_1,j_2)C_2(j_2,j_3)\cdots C_2(j_{k-1},j_k)D_2(j_k,n))\notag\\
  &=\sum_{0\leq j_1<\dots<j_k\leq n-1}\mathbb{P}(D_1(j_1))\mathbb{P}(C_2(j_1,j_2))\mathbb{P}(C_2(j_2,j_3))\cdots \mathbb{P}(C_2(j_{k-1},j_k))\mathbb{P}(D_2(j_k,n))\notag\\
  &=\sum_{0\leq j_1<\dots<j_k\leq n-1}\mathbb{P}(D_1(j_1))\mathbb{P}(C_2(j_2-j_1))\mathbb{P}(C_2(j_3-j_2))\cdots \mathbb{P}(C_2(j_k-j_{k-1}))\mathbb{P}(D_2(j_k,n)).
  \end{align}
Let $j=j_1, r=j_k-j_1, y_i=j_{i+1}-j_i, 1\leq i\leq k-1$. Then we have
\begin{align}\label{eq: alpha k general01-b}
&\mathbb{E}\alpha^+_k(n)\notag\\
&=\sum^{n-1}_{r=k-1}\left[\sum^{n-1-r}_{j=0}\mathbb{P}(D_1(j))\mathbb{P}(D_2(j+r,n))\right]\left[\sum_{\substack{0<y_i< r\\
  y_1+y_2+\dots+y_{k-1}=r}}\mathbb{P}(C_2(y_1))\mathbb{P}(C_2(y_2))\cdots \mathbb{P}(C_2(y_{k-1}))\right].
\end{align}
It follows from Lemma \ref{prop: events limit} that
there exists a positive constant $c_1$
such that for all
integers $n, r\ge 1$ and $j\ge 0$ with $n-j-r\geq 1$,
$$
\sqrt{j}\mathbb{P}(D_1(j))\leq c_1,\quad \sqrt{n-j-r}\mathbb{P}(D_2(j+r,n))\leq c_1.
$$
Thus
\begin{align}\label{eq: alpha k estimate02}
&\sum^{n-1-r}_{j=0}\mathbb{P}(D_1(j))\mathbb{P}(D_2(j+r,n))\notag\\
&=\mathbb{P}(D_2(r,n))+\sum^{n-1-r}_{j=1}\mathbb{P}(D_1(j))\mathbb{P}(D_2(j+r,n))\notag\\
&\leq c\left(1+\sum^{n-1-r}_{j=1}\frac{1}{\sqrt{j(n-r-j)}}\right)\notag\\
&=c\left(1+\sum^{n-1-r}_{j=1}\frac{1}{\sqrt{\frac{j}{n-r}}\sqrt{1-\frac{j}{n-r}}}\cdot\frac{1}{n-r}\right),
\end{align}
where $c=\max\{1, c^2_1\}$.
Since
$$
\lim\limits_{n-r\rightarrow\infty}\sum^{n-1-r}_{j=1}\frac{1}{\sqrt{\frac{j}{n-r}}\sqrt{1-\frac{j}{n-r}}}\cdot\frac{1}{n-r}=\int^1_0x^{-1/2}(1-x)^{-1/2}dx=\pi,
$$
we know that
$$
\sum^{n-1-r}_{j=1}\frac{1}{\sqrt{\frac{j}{n-r}}\sqrt{1-\frac{j}{n-r}}}\cdot\frac{1}{n-r}, \quad n-r\ge 1
$$
is bounded. Thus by \eqref{eq: alpha k estimate02},
   there exists a positive constant $C$ such that for all
      integers $n, r\ge 1$ with $n-r\geq 1$,
\begin{align}\label{3.12}
\sum^{n-1-r}_{j=0}\mathbb{P}(D_1(j))\mathbb{P}(D_2(j+r,n))\leq C.
\end{align}
Hence,  by (\ref{eq: alpha k general01-b}) and (\ref{3.12}), we have
 \begin{align*}
 \mathbb{E}\alpha^+_k(n)&\leq C\sum^{n-1}_{r=k-1}\sum_{\substack{0<y_i< r\\
  y_1+y_2+\dots+y_{k-1}=r}}\mathbb{P}(C_2(y_1))\mathbb{P}(C_2(y_2))\cdots \mathbb{P}(C_2(y_{k-1}))\\
&\leq  C\sum_{\substack{0<y_i\leq n-1\\
 i=1,2,\dots,k-1}}\mathbb{P}(C_2(y_1))\mathbb{P}(C_2(y_2))\cdots \mathbb{P}(C_2(y_{k-1}))\\
  &=C\left(\sum^{n-1}_{y=1}\mathbb{P}(C_2(y))\right)^{k-1}.
 \end{align*}

 Combining Lemma \ref{prop: events limit}(i) with Stolz's theorem, we get that
 $$
 \lim\limits_{n\rightarrow\infty}\frac{\sum^{n-1}_{y=1}\mathbb{P}(C_2(y))}{\log n}=\lim\limits_{n\rightarrow\infty}\frac{\sum^{n-1}_{y=1}
 \mathbb{P}(C_2(y))}{\sum^{n-1}_{y=1}\frac{1}{y}}\cdot \frac{\sum^{n-1}_{y=1}\frac{1}{y}}{\log n}=
 \frac12.
 $$
It follows that,  for all $\varepsilon>0$, there exists $N_1(\varepsilon)$ such that for all $n>N_1(\varepsilon)$,
$\sum^{n-1}_{y=1}\mathbb{P}(C_2(y))\leq (\frac12+\varepsilon)\log n$ and $\frac{C}{(\frac12+\varepsilon)\log n}\leq \frac{1}{2}$.
Thus for all $n> N_1(\varepsilon)$, it holds that
\begin{equation}\label{eq: alpha+k upper bound}
 \mathbb{E}\alpha^+_k(n)\leq C[(\frac12+\varepsilon)\log n]^{k-1}=
 \frac{C}{(\frac12+\varepsilon)\log n}[(\frac{1}{2}+\varepsilon)\log n]^{k}
 \leq \frac{1}{2}[(\frac12+\varepsilon)\log n]^{k}.
\end{equation}

Next, we bound $\mathbb{E}\alpha^+_k(n)$ from below.
Since $k\sim a \log n$ as $n\rightarrow\infty$, we know when $n$ is sufficiently large, $\frac{n}{3k}>1$. Let $j_1\leq \frac{n}{3}$, $0<j_l-j_{l-1}\leq \frac{n}{3k}$, $l=2,3,\dots,k$. Then $j_k=\sum^k_{l=2}(j_l-j_{l-1})+j_1<\frac{2n}{3}$. Hence, by \eqref{eq: alpha k general01}, we have that
\begin{align}\label{eq: Ealpha+k lower01}
    \mathbb{E}\alpha^+_k(n)\geq &\sum_{\substack{0\leq j_1\leq \frac{n}{3}\\ 0<j_l-j_{l-1}\leq \frac{n}{3k},\\l=2,3,\dots,k}}\mathbb{P}(D_1(j_1))\mathbb{P}(C_2(j_2-j_1))\mathbb{P}(C_2(j_3-j_2))\cdots \mathbb{P}(C_2(j_k-j_{k-1}))\mathbb{P}(D_2(j_k,n))\notag\\
    =&\sum_{k-1\leq r\leq (k-1)\frac{n}{3k}}\left[\sum_{0\leq j\leq \frac{n}{3}}\mathbb{P}(D_1(j))\mathbb{P}(D_2(j+r,n))\right]\notag\\
    &\cdot \left[\sum_{\substack{0<y_i\leq \frac{n}{3k},1\leq i\leq k-1\\ y_1+y_2+\dots+y_{k-1}=r}}\mathbb{P}(C_2(y_1))\mathbb{P}(C_2(y_2))\cdots \mathbb{P}(C_2(y_{k-1}))\right].
\end{align}
It follows from  Lemma \ref{prop: events limit}(ii)(iii)
that there exists a positive constant $c_2$ such that  for all
integers $n, r\ge 1$ and $j\ge 0$ with $n-j-r\geq1$,
$$\sqrt{j}\mathbb{P}(D_1(j))\geq c_2, \,\,\sqrt{n-j-r}\mathbb{P}(D_2(j+r,n))\geq c_2.$$
Thus
\begin{equation*}
   \sum_{0\leq j\leq \frac{n}{3}}\mathbb{P}(D_1(j))\mathbb{P}(D_2(j+r,n))\\
   \geq c^2_2\sum_{0\leq j\leq \frac{n}{3}}\frac{1}{\sqrt{nj}}=\frac{c^2_2}{\sqrt{3}}\sum_{0\leq j\leq \frac{n}{3}}\frac{1}{\sqrt{\frac{3j}{n}}}\cdot \frac{3}{n},
\end{equation*}
which together with  $\lim\limits_{n\rightarrow\infty}\sum_{0\leq j\leq \frac{n}{3}}\frac{1}{\sqrt{\frac{3j}{n}}}\cdot \frac{3}{n}=\int^1_0x^{-1/2}\mathrm{d}x=2$ implies that
there exists  a positive constant $\tilde{C}$ (independent of $r\geq 1$) such that
\begin{align}\label{eq: sum D lower}
\sum_{0\leq j\leq \frac{n}{3}}\mathbb{P}(D_1(j))\mathbb{P}(D_2(j+r,n))\geq \tilde{C}.
\end{align}
Hence, by \eqref{eq: Ealpha+k lower01} and \eqref{eq: sum D lower}, we have
\begin{align}\label{eq: Ealpha+k lower02}
\mathbb{E}\alpha^+_k(n)\geq &\tilde{C}\sum_{k-1\leq r\leq (k-1) \frac{n}{3k}}\left[\sum_{\substack{0<y_i\leq  \frac{n}{3k}, 1\leq i\leq k-1\\ y_1+y_2+\dots+y_{k-1}=r}}\mathbb{P}(C_2(y_1))\mathbb{P}(C_2(y_2))\cdots \mathbb{P}(C_2(y_{k-1}))\right]\notag\\
=&\tilde{C}\left[\sum_{0<y\leq \frac{n}{3k}}\mathbb{P}(C_2(y))\right]^{k-1}.
\end{align}
Combining the fact that $k\sim a \log n$ as $n\rightarrow\infty$ with Lemma \ref{prop: events limit}(i), we have
$$
\lim\limits_{n\rightarrow\infty}\frac{\sum_{0<y\leq \frac{n}{3k}}\mathbb{P}(C_2(y))}{\log n}=\lim\limits_{n\rightarrow\infty}\frac{\sum_{0<y\leq \frac{n}{3k}}\mathbb{P}(C_2(y))}{\sum_{0<y\leq \frac{n}{3k}}\frac{1}{y}}\cdot\frac{\sum_{0<y\leq \frac{n}{3k}}\frac{1}{y}}{\log \frac{n}{3k}}\cdot \frac{\log \frac{n}{3k}}{\log n }=
\frac12.
$$
It follows that, for any $\varepsilon\in (0, \frac12)$,
there exists $N_2(a,\varepsilon)> N_1(\varepsilon)$ such that for all $n> N_2(a,\varepsilon)$, $\sum_{0<y\leq \frac{n}{3k}}\mathbb{P}(C_2(y))\geq (\frac12-\frac{\varepsilon}{2})\log n$, which together with \eqref{eq: Ealpha+k lower02} implies that for all $n>N_2(a,\varepsilon)$,
$$
\mathbb{E}\alpha^+_k(n)\geq \tilde{C}\left[(\frac12-\frac{\varepsilon}{2})\log n\right]^{k-1}=\left[(\frac12-\varepsilon)\log n\right]^k\cdot \frac{\tilde{C}}{\frac12-\varepsilon}\left[\frac{\frac12-\frac{\varepsilon}{2}}
{\frac12-\varepsilon}\right]^{k-1}\frac{1}{\log n}.
$$
Since
$k\sim a \log n$ as $n\rightarrow\infty$, we have

$$\lim\limits_{n\rightarrow\infty}\left[\frac{\frac12-\frac{\varepsilon}{2}}{\frac12-\varepsilon}\right]
^{k-1}\frac{1}{\log n}=\infty.$$
Hence, there exists $N_3(a,\varepsilon)\geq N_2(a,\varepsilon)$ such that for all $n>N_3(a,\varepsilon)$,
\begin{equation}\label{eq: alpha+k lower bound}
  \mathbb{E}\alpha^+_k(n)\geq \frac{1}{2}\left[(\frac12-\varepsilon)\log n\right]^k.
\end{equation}

By the symmetry of $\{S_n\}_{n\geq 0}$, \eqref{eq: defn positive and negative edge}, \eqref{eq: alpha+k upper bound} and \eqref{eq: alpha+k lower bound},  we  obtain that  for all $n> N_3(a,\varepsilon)$,
\begin{equation}\label{eq: alpha positive and negative inequ}
    \frac{1}{2}\left[(\frac12-\varepsilon)\log n\right]^k\leq \mathbb{E}\alpha^-_k(n)=\mathbb{E}\alpha^+_k(n)\leq \frac{1}{2}[(\frac12+\varepsilon)\log n]^k.
\end{equation}

Under $\mathbb{P}$ (i.e. $\mathbb{P}_0$ ), if $S_n>0$, then
every edge $\langle x-1,x\rangle$ with $x\leq 0$, if has been visited, is visited at least twice,
and thus $\alpha^-(n)=0$. Similarly, if $S_n<0$, we have that $\alpha^+(n)=0$, and if $S_n=0$, then $\alpha^+(n)=\alpha^-(n)=0$. It follows that
$\alpha^+(n)\alpha^-(n)=0$.   Hence
$$
\alpha_k(n)=\alpha^+_k(n)+\alpha^-_k(n).
$$
Thus, by \eqref{eq: alpha+k upper bound}, \eqref{eq: alpha+k lower bound} and \eqref{eq: alpha positive and negative inequ},  for any $\varepsilon\in (0, \frac12)$, there exists $n_0(a,\varepsilon)=N_3(a,\varepsilon)$ such that for all $n> n_0(a,\varepsilon)$,
$$
\left[(\frac12-\varepsilon)\log n\right]^k\leq \mathbb{E}\alpha_k(n)=\mathbb{E}\alpha^+_k(n)+\mathbb{E}\alpha^-_k(n)\leq [(\frac12+\varepsilon)\log n]^k.
$$
\end{proof}

The next lemma will also play a very important role in the proof of Theorem  \ref{thm: inequality of alpha(n)}.

\begin{lem}\label{lem-3.4}
Let $\delta\in (0, \frac12)$, $\bar{a}=a(1+2\delta), d=2a$, $k=\lfloor 2a\log n\rfloor=\lfloor d\log n\rfloor$, $k'=\lfloor d(1+\delta)\log n\rfloor$, and $k''=\lfloor d(1+2\delta)\log n\rfloor$. Then
for sufficiently large $n$,
it holds that
\begin{equation}\label{eq: part estimate01}
\sum_{m\geq \bar{a}(\log n)^2}\mathbb{P}(\alpha(n)=m)\binom{m}{k'}\leq \frac{1}{3}\mathbb{E}\alpha_{k'}(n)
\end{equation}
and
\begin{equation}\label{eq: part estimate02}
	\sum_{m\leq a(\log n)^2}\mathbb{P}(\alpha(n)=m)\binom{m}{k'}\leq \frac{1}{3}\mathbb{E}\alpha_{k'}(n).
\end{equation}
\end{lem}

\begin{proof}
We will only
prove \eqref{eq: part estimate01}.
The proof of \eqref{eq: part estimate02} is similar.
For any $m\geq k''$, $\frac{\binom{m}{k'}}{\binom{m}{k''}}$ decreases as $m$ increases. Thus
\begin{align}\label{eq: part estimate03}
&\sum_{m\geq \bar{a}(\log n)^2}\mathbb{P}(\alpha(n)=m)\binom{m}{k'}\notag\\
&=\sum_{m\geq \bar{a}(\log n)^2}\mathbb{P}(\alpha(n)=m)\binom{m}{k''}\cdot\frac{\binom{m}{k'}}{\binom{m}{k''}}\notag\\
&\leq \left[\sum_{m\geq \bar{a}(\log n)^2}\mathbb{P}(\alpha(n)=m)\binom{m}{k''}\right]\cdot\frac{\binom{\lfloor \bar{a}(\log n)^2\rfloor}{k'}}{\binom{\lfloor \bar{a}(\log n)^2\rfloor}{k''}}\notag\\
&\leq \mathbb{E}\alpha_{k''}(n)\frac{\binom{\lfloor \bar{a}(\log n)^2\rfloor}{k'}}{\binom{\lfloor \bar{a}(\log n)^2\rfloor}{k''}}.
\end{align}
It follows from Lemma \ref{lem: main lem} that
that  there exists $n(\delta)$ such that for all $n>n(\delta)$,
\begin{align}\label{3.23-a}
\frac{\mathbb{E}\alpha_{k''}(n)}{\mathbb{E}\alpha_{k'}(n)}&\leq \frac{[\frac12(1+\delta^3)\log n]^{k''}}{[\frac12(1-\delta^3)\log n]^{k'}}\leq \frac{[\frac12(1+\delta^3)\log n]^{d(1+2\delta)\log n}}{[\frac12(1-\delta^3)\log n]^{d(1+\delta)\log n-1}}.
\end{align}
Using properties of  the Gamma function and Stirling's formula, we get
\begin{align}\label{3.23-b}
&\frac{\binom{\lfloor \bar{a}(\log n)^2\rfloor}{k'}}{\binom{\lfloor \bar{a}(\log n)^2\rfloor}{k''}}=\frac{\lfloor \bar{a}(\log n)^2\rfloor !}{k'!(\lfloor \bar{a}(\log n)^2\rfloor-k')!}\cdot
\frac{k''!(\lfloor \bar{a}(\log n)^2\rfloor-k'')!}{\lfloor \bar{a}(\log n)^2\rfloor !}\notag\\
&=\frac{\Gamma(k''+1)\Gamma(\lfloor \bar{a}(\log n)^2\rfloor-k''+1)}{\Gamma(k'+1)\Gamma(\lfloor \bar{a}(\log n)^2\rfloor-k'+1)}\notag\\
&\leq \frac{\Gamma(d(1+2\delta)\log n+1)\Gamma(\bar{a}(\log n)^2-d(1+2\delta)\log n+2)}{\Gamma(d(1+\delta)\log n)\Gamma(\bar{a}(\log n)^2-d(1+\delta)\log n)}\notag\\
&=[d(1+\delta)\log n]\cdot [\bar{a}(\log n)^2-d(1+\delta)\log n]\cdot [\bar{a}(\log n)^2-d(1+2\delta)\log n+1]\notag\\
&\quad \cdot\frac{\Gamma(d(1+2\delta)\log n+1)\Gamma(\bar{a}(\log n)^2-d(1+2\delta)\log n+1)}{\Gamma(d(1+\delta)\log n+1)\Gamma(\bar{a}(\log n)^2-d(1+\delta)\log n+1)}\notag\\
&\leq C(\log n)^5\notag\\
	&\quad\cdot \frac{(d(1+2\delta)\log n)^{d(1+2\delta)\log n+1/2}\cdot (\bar{a}(\log n)^2-d(1+2\delta)\log n)^{\bar{a}(\log n)^2-d(1+2\delta)\log n+1/2}}{(d(1+\delta)\log n)^{d(1+\delta)\log n+1/2}\cdot (\bar{a}(\log n)^2-d(1+\delta)\log n)^{\bar{a}(\log n)^2-d(1+\delta)\log n+1/2}}.
\end{align}
Combining (\ref{3.23-a}) and (\ref{3.23-b}), we get that for all $n>n(\delta)$,
\begin{align}\label{eq: alpha k' and alpha k''01}
&\frac{\mathbb{E}\alpha_{k''}(n)}{\mathbb{E}\alpha_{k'}(n)}\cdot \frac{\binom{\lfloor \bar{a}(\log n)^2\rfloor}{k'}}{\binom{\lfloor \bar{a}(\log n)^2\rfloor}{k''}}\notag\\
&\leq C\frac{[\frac12(1+\delta^3)\log n]^{d(1+2\delta)\log n}}{[\frac12(1-\delta^3)\log n]^{d(1+\delta)\log n-1}}(\log n)^5\notag\\
	&\quad\cdot \frac{(d(1+2\delta)\log n)^{d(1+2\delta)\log n+1/2}\cdot (\bar{a}(\log n)^2-d(1+2\delta)\log n)^{\bar{a}(\log n)^2-d(1+2\delta)\log n+1/2}}{(d(1+\delta)\log n)^{d(1+\delta)\log n+1/2}\cdot (\bar{a}(\log n)^2-d(1+\delta)\log n)^{\bar{a}(\log n)^2-d(1+\delta)\log n+1/2}}\notag\\
&\leq C(\log n)^6\frac{(1+\delta^3)^{d(1+2\delta)\log n}}{(1-\delta^3)^{d(1+\delta)\log n}}\left(\frac{1+2\delta}{1+\delta}\right)^{d(1+\delta)\log n}\frac{\left(1-\frac{d(1+2\delta)}{\bar{a}\log n}\right)^{\bar{a}(\log n)^2-d(1+2\delta)\log n+1/2}}{\left(1-\frac{d(1+\delta)}{\bar{a}\log n}\right)^{\bar{a}(\log n)^2-d(1+\delta)\log n+1/2}}.
\end{align}
By Taylor's expansion, we have
\begin{align*}
&\left(1-\frac{d(1+2\delta)}{\bar{a}\log n}\right)^{\bar{a}(\log n)^2-d(1+2\delta)\log n+1/2}\\
&=\exp\left\{\left[\bar{a}(\log n)^2-d(1+2\delta)\log n+1/2\right]\log\left(1-\frac{d(1+2\delta)}{\bar{a}\log n}\right)\right\}\\
&=\exp\left\{\left[\bar{a}(\log n)^2-d(1+2\delta)\log n+1/2\right]\left[-\frac{d(1+2\delta)}{\bar{a}\log n}+O((\log n)^{-2})\right]\right\}\\
&=\exp\{-d(1+2\delta)\log n+O(1)\}.
\end{align*}
Similarly, we have
\begin{equation*}
\left(1-\frac{d(1+\delta)}{\bar{a}\log n}\right)^{\bar{a}(\log n)^2-d(1+\delta)\log n+1/2}=\exp\{-d(1+\delta)\log n+O(1)\},
\end{equation*}
\begin{align*}
\left(\frac{1+2\delta}{1+\delta}\right)^{d(1+\delta)\log n}=&\exp\left\{d(1+\delta)[\log(1+2\delta)-\log(1+\delta)]\log n\right\}\\
=&\exp\left\{d(1+\delta)(\delta-\frac{3}{2}\delta^2+O(\delta^3))\log n\right\}\\
=&\exp\left\{\left(d\delta-\frac{d\delta^2}{2}+O(\delta^3)\right)\log n\right\},
\end{align*}
and
\begin{align*}
&\frac{(1+\delta^3)^{d(1+2\delta)\log n}}{(1-\delta^3)^{d(1+\delta)\log n}}
=\exp\left\{d(1+2\delta)\log(1+\delta^3)\log n-d(1+\delta)\log(1-\delta^3)\log n\right\}\notag\\
&\hspace{3.2cm}=\exp\{O(\delta^3)\log n\}.
\end{align*}
Combining the four displays above with \eqref{eq: alpha k' and alpha k''01}, we get
\begin{align*}
&\frac{\mathbb{E}\alpha_{k''}(n)}{\mathbb{E}\alpha_{k'}(n)}\cdot \frac{\binom{\lfloor \bar{a}(\log n)^2\rfloor}{k'}}{\binom{\lfloor \bar{a}(\log n)^2\rfloor}{k''}}\\
&\leq C(\log n)^6\cdot\exp\{O(\delta^3)\log n\}\cdot\exp\left\{\left(d\delta-\frac{d\delta^2}{2}+O(\delta^3)\right)\log n\right\}\cdot\frac{\exp\{-d(1+2\delta)\log n\}}{\exp\{-d(1+\delta)\log n\}}\\
&\leq \exp\left\{-\left(\frac{d}{2}-O(\delta)\right)\delta^2\log n+6\log\log n+C\right\}.
\end{align*}
Thus there exists $\delta(d)>0$ such that for all $\delta<\delta(d)$, we have $\frac{d}{2}-O(\delta)=a-O(\delta)>0$. Hence we have
\begin{equation}\label{eq: eq: alpha k' and alpha k'' limit}
\lim\limits_{n\rightarrow\infty}\frac{\mathbb{E}\alpha_{k''}(n)}{\mathbb{E}\alpha_{k'}(n)}\cdot \frac{\binom{\lfloor \bar{a}(\log n)^2\rfloor}{k'}}{\binom{\lfloor \bar{a}(\log n)^2\rfloor}{k''}}=0.
\end{equation}
Combining this with \eqref{eq: part estimate03}, we get that, for any fixed $\delta\in (0, \delta(d))$,  there exists $N(\delta)>n(\delta)$ such that for all $n> N(\delta)$,
$$	\sum_{m\geq \bar{a}(\log n)^2}\mathbb{P}(\alpha(n)=m)\binom{m}{k'}\leq \mathbb{E}\alpha_{k'}(n)\cdot \frac{\mathbb{E}\alpha_{k''}(n)}{\mathbb{E}\alpha_{k'}(n)}\cdot \frac{\binom{\lfloor \bar{a}(\log n)^2\rfloor}{k'}}{\binom{\lfloor \bar{a}(\log n)^2\rfloor}{k''}}\leq \frac{1}{3}\mathbb{E}\alpha_{k'}(n).
$$
\end{proof}

The following two lemmas are important for proving Theorem~\ref{thm: limitation of alpha(n)}.

\begin{lem}\label{lem: limit}
Let $\alpha(n)$ be the same as in \eqref{eq: defn edge}.
If $\{f(n)\}_{n\ge 1}$ satisfies $0<f(n)<n$ and $\lim\limits_{n\to\infty}f(n)=\infty$, then there exists $C\in [0, \infty]$
such that $\mathbb{P}\left(\limsup\limits_{n\rightarrow\infty}\frac{\alpha(n)}{f(n)}=C\right)=1$.
\end{lem}
The proof of Lemma \ref{lem: limit} is routine by Kolmogorov's 0-1 law.
For the reader's convenience,
we put the detail of the proof in the appendix.

\begin{lem}\label{lem: stopping time}
Let $\sigma_n:=\inf\{k\geq 0: S_k=n\}$ for all $n\geq 0$. Then for any $q>2$,
$$\mathbb{P}(\lim_{n\to+\infty}\frac{\sigma_n}{n^q})=1.$$
\end{lem}
The proof of Lemma \ref{lem: stopping time} is routine.
For the reader's convenience, we put the detail of the proof in the appendix.

\subsection{Proof of Theorem \ref{thm: inequality of alpha(n)}}

	In this proof, $C$ stands for a positive constant whose value may change from one appearance to another. We prove the theorem in two steps.
	
{\bf Step 1}:
In this step, we will prove that, for all $a>0$ and $\varepsilon>0$, there exists $N_1(a,\varepsilon)$ such that for all $n>N_1(a,\varepsilon)$,
	$$\mathbb{P}(\alpha(n)>a(\log n)^2)<n^{-2a+\varepsilon}.$$
Let $k=\lfloor{2a}\log n\rfloor$ and $0<\delta<1$. By Markov's  inequality, we have
\begin{align}\label{proof-thm1.2-a}
\mathbb{P}\left(\alpha(n)>a(\log n)^2\right)
&\leq \mathbb{P}\left(\alpha(n)>\lfloor a(\log n)^2\rfloor\right)\notag\\
&=\mathbb{P}\left(\alpha_k(n)>\binom{\lfloor a(\log n)^2\rfloor}{k}\right)\notag\\
&\leq  \frac{\mathbb{E}\alpha_k(n)}{\binom{\lfloor a(\log n)^2\rfloor}{k}}.
\end{align}
By properties of the Gamma function, we have that  for $k=\lfloor{2a}\log n\rfloor$,
\begin{align}\label{proof-thm1.2-b}
\frac{1}{\binom{\lfloor a(\log n)^2\rfloor}{k}}
&=\frac{(\lfloor a(\log n)^2\rfloor-k)!k!}{\lfloor a(\log n)^2\rfloor!}\notag\\
&=\frac{\Gamma(k+1)\Gamma(\lfloor a(\log n)^2\rfloor-k+1)}{\Gamma(\lfloor a(\log n)^2\rfloor+1)}\notag\\
&\leq \frac{\Gamma(2a\log n+1)\Gamma(a(\log n)^2-2a\log n+2)}{\Gamma(a(\log n)^2)}\notag\\
&=a(\log n)^2\left(a(\log n)^2-2a\log n+1\right)\notag\\
&\quad\cdot \frac{\Gamma(2a\log n+1)\Gamma(a(\log n)^2-2a\log n+1)}{\Gamma(a(\log n)^2+1)}.
\end{align}
By (\ref{proof-thm1.2-a}), (\ref{proof-thm1.2-b}) and  Lemma \ref{lem: main lem},
 there exists $n_1(\delta)$ such that for all $n\geq n_1(\delta) $,
\begin{align}\label{eq: alpha n ineuq upper bound01}
\mathbb{P}\left(\alpha(n)>a(\log n)^2\right)
&\leq [\frac12(1+\delta)\log n]^{2a\log n}\cdot a(\log n)^2\left(a(\log n)^2-2a\log n+1\right)\notag\\
&\quad\cdot \frac{\Gamma(2a\log n+1)\Gamma(a(\log n)^2-2a\log n+1)}{\Gamma(a(\log n)^2+1)}.
\end{align}
Then by Stirling's formula, we have
\begin{align}\label{eq: alpha n ineuq upper bound02}
&\mathbb{P}\left(\alpha(n)>a(\log n)^2\right)\notag\\
&\leq C[\frac12(1+\delta)\log n]^{2a\log n}\cdot (\log n)^4\notag\\
&\quad\cdot \frac{(2a\log n)^{2a\log n+1/2}[a(\log n)^2-2a\log n]^{a(\log n)^2-2a\log n+1/2}}{[a(\log n)^2]^{a(\log n)^2+1/2}}\notag\\
&\leq C(1+\delta)^{2a\log n}\cdot (\log n)^{9/2}\left[1-\frac{2}{\log n}\right]^{a(\log n)^2-2a\log n+1/2}.
\end{align}
By Taylor's expansion, we have
\begin{align*}
&\left[1-\frac{2}{\log n}\right]^{a(\log n)^2-2a\log n+1/2}\\
&=\exp\left\{\left[a(\log n)^2-2a\log n+1/2\right]\log\left(1-\frac{2}{\log n}\right)\right\}\\
&=\exp\left\{\left[a(\log n)^2-2a\log n+1/2\right]\left[-\frac{2}{\log n}+O((\log n)^{-2})\right]\right\}\\
&=\exp\left\{-2a(\log n)+O(1)\right\},
\end{align*}
which together with (\ref{eq: alpha n ineuq upper bound02}) implies  that for all $\varepsilon >0$,
\begin{align*}
&\mathbb{P}(\alpha(n)>a(\log n)^2)\\
&\leq C(1+\delta)^{2a\log n}(\log n)^{\frac{9}{2}}\cdot \exp\left\{-2a(\log n)\right\}\\
&=n^{-2a+\varepsilon}\cdot \exp\left\{\left[2a\log(1+\delta)-\varepsilon\right]\log n+\frac{9}{2}\log(\log n)+C\right\}.
\end{align*}
 Hence, for any fixed $\varepsilon>0$, there exists $0<\delta_1(\varepsilon)<\frac{1}{2}$ such that for all $\delta<\delta_1(\varepsilon)$, $2a\log(1+\delta)-\varepsilon<0$. Thus we have $\lim\limits_{n\rightarrow\infty}\exp\{[2a\log(1+\delta)-\varepsilon]\log n+\frac{9}{2}\log(\log n)+C\}=0$.
 Therefore,
 for any $\delta\in (0, \delta_1(\varepsilon))$,
 there exists $N_1(a,\varepsilon)>n_1(\delta)$ such that  for all $n>N_1(a,\varepsilon)$,
 \begin{equation}\label{eq: alpha n upper}
\mathbb{P}(\alpha(n)>a(\log n)^2)<n^{-2a+\varepsilon}.
 \end{equation}

{\bf Step 2}: In this step, we will prove that for all $a>0$ and $\varepsilon>0$, when $n$ is sufficiently large,
\begin{align}\label{Step-2-a}
\mathbb{P}(\alpha(n)>a(\log n)^2)>n^{-2a-\varepsilon}.
\end{align}

We will use the notation $\delta, {\bar a}, d, k, k'$ and $k''$ in the statement of Lemma \ref{lem-3.4}.
It follows from Lemma \ref{lem-3.4} that for sufficiently large $n$,
\begin{align}\label{eq: alpha n lower}
&\mathbb{P}(\alpha(n)>a(\log n)^2)\notag\\
&\geq  \sum_{a(\log n)^2<m<\bar{a}(\log n)^2}\mathbb{P}(\alpha(n)=m)\notag\\
&\geq \frac{1}{\binom{\lfloor \bar{a}(\log n)^2\rfloor}{k'}} \sum_{a(\log n)^2<m<\bar{a}(\log n)^2}\mathbb{P}(\alpha(n)=m)\binom{m}{k'}\notag\\
	&= \frac{1}{\binom{\lfloor \bar{a}(\log n)^2\rfloor}{k'}}\left[\mathbb{E}\alpha_{k'}(n)-\sum_{m\geq \bar{a}(\log n)^2}\mathbb{P}(\alpha(n)=m)\binom{m}{k'}-\sum_{m\leq a(\log n)^2}\mathbb{P}(\alpha(n)=m)\binom{m}{k'}\right]\notag\\
	&\geq \frac{1}{3}\frac{\mathbb{E}\alpha_{k'}(n)}{\binom{\lfloor \bar{a}(\log n)^2\rfloor}{k'}}.
\end{align}

Now we focus on the quantity
$\frac{\mathbb{E}\alpha_{k'}(n)}{\binom{\lfloor \bar{a}(\log n)^2\rfloor}{k'}}$.   By  Lemma \ref{lem: main lem} and
properties of  the Gamma function,
there exists $n_2(\delta)$ such that for all $n>n_2(\delta)$,
\begin{align}\label{eq: alpha k' lower01}
&\frac{\mathbb{E}\alpha_{k'}(n)}{\binom{\lfloor \bar{a}(\log n)^2\rfloor}{k'}}\notag\\
&\geq [\frac12(1-\delta)\log n]^{k'}\frac{\Gamma(k'+1)\Gamma(\lfloor \bar{a}(\log n)^2\rfloor-k'+1)}{\Gamma(\lfloor \bar{a}(\log n)^2\rfloor+1)}\notag\\
&\geq [\frac12(1-\delta)\log n]^{d(1+\delta)\log n-1}\frac{\Gamma(d(1+\delta)\log n)\Gamma(\bar{a}(\log n)^2-d(1+\delta)\log n)}{\Gamma(\bar{a}(\log n)^2+1)}\notag\\
&\geq [\frac12(1-\delta)\log n]^{d(1+\delta)\log n-1}\left[d(1+\delta)\log n\right]^{-1}\left[\bar{a}(\log n)^2-d(1+\delta)\log n\right]^{-1}\notag\\
&\quad\cdot \frac{\Gamma(d(1+\delta)\log n+1)\Gamma(\bar{a}(\log n)^2-d(1+\delta)\log n+1)}{\Gamma(\bar{a}(\log n)^2+1)}\notag\\
&\geq C[\frac12(1-\delta)\log n]^{d(1+\delta)\log n}(\log n)^{-4}\notag\\
&\quad\cdot\frac{(d(1+\delta)\log n)^{d(1+\delta)\log n+1/2}(\bar{a}(\log n)^2-d(1+\delta)\log n)^{\bar{a}(\log n)^2-d(1+\delta)\log n+1/2}}{(\bar{a}(\log n)^2)^{\bar{a}(\log n)^2+1/2}}\notag\\
&\geq C(\log n)^{-7/2}\left(\frac{1-\delta}{1+2\delta}\right)^{d(1+\delta)\log n}\left(1-\frac{d(1+\delta)}{\bar{a}\log n}\right)^{\bar{a}(\log n)^2-d(1+\delta)\log n+1/2}.
\end{align}
By Taylor's expansion, we have
\begin{align*}
&\left(1-\frac{d(1+\delta)}{\bar{a}\log n}\right)^{\bar{a}(\log n)^2-d(1+\delta)\log n+1/2}\\
&=\exp\left\{\left[\bar{a}(\log n)^2-d(1+\delta)\log n+1/2\right]\log\left(1-\frac{d(1+\delta)}{\bar{a}\log n}\right)\right\}\\
	&=\exp\left\{\left[\bar{a}(\log n)^2-d(1+\delta)\log n+1/2\right]\left[-\frac{d(1+\delta)}{\bar{a}\log n}+O((\log n)^{-2})\right]\right\}\\
	&=\exp\left\{-d(1+\delta)\log n+O(1)\right\}
\end{align*}
and
\begin{align*}
	\left(\frac{1-\delta}{1+2\delta}\right)^{d(1+\delta)\log n}=&\exp\left\{d(1+\delta)[\log(1-\delta)-\log(1+2\delta)]\log n\right\}\\
	=&\exp\left\{d(1+\delta)(-3\delta+O(\delta^2))\log n\right\}\\
=&\exp\left\{d(-3\delta+O(\delta^2))\log n\right\}.
\end{align*}
Combining the two displays above with \eqref{eq: alpha k' lower01}, we get that for any $\varepsilon>0$,
\begin{align*}
	\frac{\mathbb{E}\alpha_{k'}(n)}{\binom{\lfloor \bar{a}(\log n)^2\rfloor}{k'}}\geq& C(\log n)^{-7/2}\cdot\exp\left\{d(-3\delta+O(\delta^2))\log n\right\}\cdot\exp\left\{-d(1+\delta)\log n\right\}\\
	=&n^{-(d+\varepsilon)}\cdot\exp\left\{C-\frac{7}{2}\log(\log n)+(\varepsilon-4d\delta+O(\delta^2))\log n\right\}.
\end{align*}
For any $\varepsilon>0$, there exists $\delta_2(\varepsilon)>0$ such that
$\varepsilon-4d\delta+O(\delta^2)>0$ for all $\delta<\delta_2(\varepsilon)$. Thus
$\lim\limits_{n\rightarrow\infty}\exp\left\{C-\frac{7}{2}\log\log n+(\varepsilon-4d\delta+O(\delta^2))\log n\right\}=\infty$. Hence, there exists $N_2(\delta)>n_2(\delta)$ such that for all $n>N_2(\delta)$,
\begin{equation}\label{eq: E alpha lower bound}
\frac{\mathbb{E}\alpha_{k'}(n)}{\binom{\lfloor \bar{a}(\log n)^2\rfloor}{k'}}>3n^{-(d+\varepsilon)}=3n^{-2a-\varepsilon},
\end{equation}
which together with (\ref{eq: alpha n lower}) implies that (\ref{Step-2-a}) holds.\hfill\fbox

\begin{rem}\label{rem: alpha positive}
Lemma \ref{lem: main lem}, which gives upper and lower bounds for $\mathbb{E}\alpha_k(n)$, played a key role in the proof of Theorem \ref{thm: inequality of alpha(n)}.
By \eqref{eq: alpha positive and negative inequ} we know that $\mathbb{E}\alpha^+_k(n)$
similar upper and lower bounds.
So by following the proof of Theorem \ref{thm: inequality of alpha(n)}, we can get
the conclusion of Theorem \ref{thm: inequality of alpha(n)} holds with $\alpha(n)$ replaced by $\alpha^+(n)$,
i.e. for all $a>0$ and $\varepsilon>0$, there exists an $N_0=N_0(a,\varepsilon)$ such that for all $n>N_0$,
$$
n^{-2a-\varepsilon}<\mathbb{P}\left(\alpha^+(n)>a(\log n)^2\right)<n^{-2a+\varepsilon}.
$$
\end{rem}

\subsection{Proof of Theorem \ref{thm: limitation of alpha(n)}}

	{\bf Step 1.}
	First we deal with the upper bound of $\limsup_{n\rightarrow\infty}\frac{\alpha(n)}{(\log n)^2}$.
	
	By Theorem \ref{thm: inequality of alpha(n)}, for any $\varepsilon>0$, there exists a positive integer $n_0$ such that for all $n>n_0$,
$$
\mathbb{P}\left(\alpha(n)>(\frac12+\varepsilon)(\log  n)^2\right)<n^{-(1+2\varepsilon)+\varepsilon}=n^{-1-\varepsilon}.
$$
It follows that
$$
\sum^{\infty}_{n=1}\mathbb{P}\left(\alpha(n)>(\frac12+\varepsilon)(\log  n)^2\right)\leq \sum^{n_0}_{n=1}\mathbb{P}\left(\alpha(n)>(\frac12+\varepsilon)(\log n)^2 \right)+\sum^{\infty}_{n=n_0+1}n^{-1-\varepsilon}<\infty.
$$
Thus by the Borel-Cantelli lemma, we have
$$
\mathbb{P}\left(\limsup\limits_{n\rightarrow\infty}\frac{\alpha(n)}{(\log n)^2}\leq \frac12+\varepsilon\right)\geq \mathbb{P}\left(\bigcup^{\infty}_{k=1}\bigcap^{\infty}_{n=k}\left\{\alpha(n)\leq (\frac12+\varepsilon)(\log n)^2 \right\}\right)=1.
$$
Hence
\begin{equation}\label{eq: alpha n upper bound}
\mathbb{P}\left(\limsup\limits_{n\rightarrow\infty}\frac{\alpha(n)}{(\log n)^2}\leq \frac12\right)=\lim\limits_{\varepsilon\rightarrow 0^+}\mathbb{P}\left(\limsup\limits_{n\rightarrow\infty}\frac{\alpha(n)}{(\log n)^2}\leq \frac12+\varepsilon\right)=1.
	\end{equation}

{\bf Step 2.}
In this step, we deal with the lower bound of  $\limsup_{n\rightarrow\infty}\frac{\alpha(n)}{(\log n)^2}$.

Recall that $\mathcal{A}^+_n$ is defined in \eqref{eq: defn positive and negative edge}.
For $k\geq 1$,  define
$$
\mathcal{A}^{+}(\sigma_{k^2}, \sigma_{k^2}+k):=\{z\geq k^2, \,\,z\in\mathcal{A}^+_{\sigma_{k^2}+k}\},
$$
where $\sigma_{k^2}:=\inf\{n\geq 0: S_n=k^2\}$.  Then we have
\begin{align*}
\mathcal{A}^{+}(\sigma_{k^2}, \sigma_{k^2}+k)&=\{z\geq k^2: \exists! t\in [\sigma_{k^2},\sigma_{k^2}+k)\ s.t.\ S_t=z, S_{t+1}=z+1\}\\
&\in\sigma(X_{\sigma_{k^2}+1}, X_{\sigma_{k^2}+2}, \dots, X_{\sigma_{k^2}+k}).
\end{align*}
Since $\sigma_{(k+1)^2}-\sigma_{k^2}\geq 2k+1$, we get that $\{\mathcal{A}^{+}(\sigma_{k^2}, \sigma_{k^2}+k), k\geq 1\}$ are independent.

We define $\tilde{S}^{\sigma_{k^2}}_t:=S_{\sigma_{k^2}+t}-S_{\sigma_{k^2}}, 0\leq t\leq k$. Then $\{\tilde{S}^{\sigma_{k^2}}_t\}_{0\leq t\leq k}$ is a simple symmtric random walk with $\tilde{S}^{\sigma_{k^2}}_0=0$. We denote $\tilde{\mathcal{A}}^{\sigma_{k^2},+}(k)$ the counterpart of $\mathcal{A}^+_k$ in \eqref{eq: defn positive and negative edge} for  the random walk $\{\tilde{S}^{\sigma_{k^2}}_t\}_{0\leq t\leq k}$. Then we know that $\#\mathcal{A}^{+}(\sigma_{k^2}, \sigma_{k^2}+k)$ and $\#\tilde{\mathcal{A}}^{\sigma_{k^2},+}(k)$ have the same distribution.   Remark \ref{rem: alpha positive} tells us that  Theorem \ref{thm: inequality of alpha(n)} also holds for $\#\tilde{\mathcal{A}}^{\sigma_{k^2},+}(k)$. Hence, for all $\varepsilon\in (0, \frac12)$, we have
\begin{align*}
&\sum^{\infty}_{k=1}\mathbb{P}\left(\#\mathcal{A}^{+}(\sigma_{k^2}, \sigma_{k^2}+k)>(\frac12-\varepsilon)(\log k)^2\right)\\
&=\sum^{\infty}_{k=1}\mathbb{P}\left(\#\tilde{\mathcal{A}}^{\sigma_{k^2},+}(k)>(\frac12-\varepsilon)(\log k)^2\right)\\
&\geq \sum^{k_0}_{k=1}\mathbb{P}\left(\#\tilde{\mathcal{A}}^{\sigma_{k^2},+}(k)>(\frac12-\varepsilon)(\log k)^2 \right)+\sum^{\infty}_{k=k_0+1}k^{-2(\frac{1}{2}-\varepsilon)-\varepsilon}
=\infty.
\end{align*}
Then,  by the Borel-Cantelli lemma again, we get
$$
\mathbb{P}\left(\#\mathcal{A}^{+}(\sigma_{k^2}, \sigma_{k^2}+k)>(\frac12-\varepsilon)(\log k)^2, i.o.\right)=1,
$$
which together with the fact that $\#\mathcal{A}^{+}(\sigma_{k^2}, \sigma_{k^2}+k)\leq \alpha^+(\sigma_{k^2}+k)$ implies that
$$
\mathbb{P}\left(\alpha^+(\sigma_{k^2}+k)>(\frac12-\varepsilon)(\log k)^2, i.o.\right)\geq \mathbb{P}(\#\mathcal{A}^{+}(\sigma_{k^2}, \sigma_{k^2}+k)>(\frac12-\varepsilon)(\log k)^2, i.o.)=1.
$$
Then by Lemma~\ref{lem: stopping time} we have that for any $q>2$,
\begin{equation}\label{eq: a.e.01}
	\mathbb{P}\left(\left\{\alpha^+(\sigma_{k^2}+k)>(\frac12-\varepsilon)(\log k)^2, i.o.\right\}\cap\left\{\lim_{n\rightarrow\infty}\frac{\sigma_n}{n^q}=0\right\}\right)=1.
\end{equation}

For any $\omega\in\{\alpha^+(\sigma_{k^2}+k)>(\frac12-\varepsilon)(\log k)^2, i.o.\}\cap\{\lim\limits_{n\rightarrow\infty}\frac{\sigma_n}{n^q}=0\}$, there exists $k_j(\omega)\rightarrow\infty$, as $j\rightarrow\infty$ such that,
for all $j\ge 1$, $\frac{\alpha^+(\sigma_{k^2_j}+k_j)}{(\log k_j)^2}>\frac12-\varepsilon$, $\sigma_{k^2_j}< \frac{1}{2}k_j^{2q}$ and $k_j<\frac{1}{2}k_j^{2q}$. Thus
\begin{align*}
\frac{\alpha(\sigma_{k^2_j}+k_j)}{[\log(\sigma_{k^2_j}+k_j)]^2}
&\geq  \frac{\alpha^+(\sigma_{k^2_j}+k_j)}{[\log(\sigma_{k^2_j}+k_j)]^2}
=\frac{\alpha^+(\sigma_{k^2_j}+k_j)}{(\log k_j)^2}\cdot \frac{(\log k_j)^2}{[\log(\sigma_{k^2_j}+k_j)]^2}\\
&>(\frac12-\varepsilon)\frac{(\log k_j)^2}{(\log k^{2q}_j)^2}
=\frac{1-2\varepsilon}{8q^2}.
\end{align*}
Hence, we have
$$
\limsup\limits_{n\rightarrow\infty}\frac{\alpha(n)(\omega)}{(\log n)^2}\geq \limsup\limits_{j\rightarrow\infty}\frac{\alpha(\sigma_{k^2_j}+k_j)}
{[\log(\sigma_{k^2_j}+k_j)]^2}(\omega)\geq \frac{1-2\varepsilon}{8q^2}.
$$
So by  (\ref{eq: a.e.01}), we have
$$
\mathbb{P}\left(\limsup\limits_{n\rightarrow\infty}\frac{\alpha(n)}{(\log n)^2}\geq \frac{1-2\varepsilon}{8q^2}\right)=1.
$$
Thus we have
\begin{equation*}\label{eq: alpha n lower bound}
	\mathbb{P}\left(\limsup_{n\rightarrow\infty}\frac{\alpha(n)}{(\log n)^2}\geq \frac{1}{8q^2}\right)=\lim\limits_{\varepsilon\rightarrow 0^+}\mathbb{P}\left(\limsup\limits_{n\rightarrow\infty}\frac{\alpha(n)}{(\log n)^2}\geq \frac{1-2\varepsilon}{8q^2}\right)=1.
\end{equation*}

Since $q>2$ is arbitrary, we have
\begin{equation}\label{eq: alpha n lower bound02}
	\mathbb{P}\left(\limsup_{n\rightarrow\infty}\frac{\alpha(n)}{(\log n)^2}> \frac{1}{32}\right)=\lim_{q\to 2^+}\mathbb{P}\left(\limsup_{n\rightarrow\infty}\frac{\alpha(n)}{(\log n)^2}\geq \frac{1}{8q^2}\right)=1.
\end{equation}

Hence, by  \eqref{eq: alpha n upper bound} we get
\begin{equation}\label{eq: alpha n bound}
	\mathbb{P}\left(\frac{1}{32}< \limsup\limits_{n\rightarrow\infty}\frac{\alpha(n)}{(\log n)^2}\leq \frac12\right)=1.
\end{equation}
Hence, by Lemma~\ref{lem: limit} and (\ref{eq: alpha n bound}),  we know that there exists a constant $C\in \left(\frac{1}{32}, \frac12\right]$ such that
$$
\mathbb{P}\left(\limsup_{n\rightarrow\infty}\frac{\alpha(n)}{(\log n)^2}=C\right)=1.
$$
The proof is complete. \hfill\fbox

\begin{rem}\label{rem-3.6}

Recently,  Feng and Hao \cite{FH25}
improved the result of \cite{Ma88} and proved that
$$
\limsup_{n\to\infty}\frac{f_1(n)}{(\log n)^2}=\frac{1}{16}  
$$
almost surely, where $f_1(n)$ is the number of rarely-visited sites up to time $n$ defined in \eqref{e:raresites}. We
believe
that the constant $C$ in Theorem \ref{thm: limitation of alpha(n)}
is  also $1/16$.
\end{rem}

\begin{rem}\label{rem-3.7}
The basic idea for our proofs of Theorems 1.2 and 1.3 comes from Major \cite{Ma88}.
The main difference is between Lemma \ref{lem: main lem} and its counterpart in  \cite{Ma88}.
The proof of Lemma \ref{lem: main lem} is more complicated than the proof of  its counterpart in  \cite{Ma88}. We have to use the 5 events $C_1(t), C_2(r,t), C_2(t), D_1(t)$ and $D_2(r,t)$ defined before Lemma \ref{prop: events limit} to prove Lemma \ref{lem: main lem}, while in \cite{Ma88}, only the following 3 events are needed:
\begin{align*}
\tilde{C}_1(r,t)&=\{\omega: S_r(\omega)<S_l(\omega)<S_t(\omega)\ \mbox{for all}\ f<l<t\},\\
\tilde{D}_1(j)&=\{\omega: S_l(\omega)<S_j(\omega), \mbox{for all}\ 0\leq l<j\},\\
\tilde{D}_2(j)&=\tilde{D}_2(j,n)=\{\omega: S_l(\omega)>S_j(\omega)\ \mbox{for all}\ j<l\leq n\}.
\end{align*}

In our proof of Theorem \ref{thm: limitation of alpha(n)}, by Lemma~\ref{lem: stopping time}, we can get the constant $C\in(\frac{1}{32},\frac{1}{2}]$. However,  in the proof of  Major \cite[Theorem 1]{Ma88}, the auther used
the fact $\mathbb{P}(\lim_{n\to+\infty}\frac{\sigma_n}{n^4}=0)=1$, which leads a larger range of the constant $C$.

\end{rem}

\bigskip

{ \noindent {\bf\large Appendix: Proof of Lemma~\ref{lem: limit} and Lemma \ref{lem: stopping time}}\quad

\begin{proof}[\bf Proof of Lemma \ref{lem: limit}.]
Let $\alpha'(n)$ denote the number of edges visited only once by the random walk $\{S_k\}_{k\geq 0}$ during the time $[\sqrt{f(n)}, n]$. Then $|\alpha(n)-\alpha'(n)|\leq 2\sqrt{f(n)}$. Hence,
\begin{equation}\label{eq: alpha n and alpha'n}
\mathbb{P}\left(\limsup_{n\rightarrow+\infty}\frac{\alpha(n)}{f(n)}=\limsup_{n\rightarrow+\infty}\frac{\alpha'(n)}{f(n)}\right)=1.
\end{equation}

For any $c\in[0,+\infty)$, $\{\limsup_{n\rightarrow+\infty}\frac{\alpha'(n)}{f(n)}\geq c\}$ is a tail event.
Hence, by  Kolmogorov’s 0-1 law,
	$$\mathbb{P}(\limsup_{n\rightarrow+\infty}\frac{\alpha'(n)}{f(n)}\geq c)=1\,\,\text{or}\,\, 0, \,\,\text{for all} \,\,c\in[0,+\infty).$$
Notice that $\mathbb{P}(\limsup_{n\rightarrow+\infty}\frac{\alpha'(n)}{f(n)}\geq0)=1$ and $\{\limsup_{n\rightarrow+\infty}\frac{\alpha'(n)}{f(n)}\geq c\}$ decreases as $c$ increases. Let
	$$c^*:=\sup\{c\geq 0: \mathbb{P}(\limsup_{n\rightarrow+\infty}\frac{\alpha'(n)}{f(n)}\geq c)=1\}.$$
	
	If $0\leq c^*<+\infty$,
		let $c_m\downarrow c^*$ as $m\rightarrow+\infty$ and get
	$$\mathbb{P}(\limsup_{n\rightarrow+\infty}\frac{\alpha'(n)}{f(n)}>c^*)=\mathbb{P}(\lim_{m\rightarrow+\infty}\{\limsup_{n\rightarrow+\infty}\frac{\alpha'(n)}{f(n)}\geq c_m\})=\lim_{m\rightarrow+\infty}\mathbb{P}(\limsup_{n\rightarrow+\infty}\frac{\alpha'(n)}{f(n)}\geq c_m)=0.$$
	Hence,
	\begin{equation}\label{eq: limsup01}
		\mathbb{P}(\limsup_{n\rightarrow+\infty}\frac{\alpha'(n)}{f(n)}\leq c^*)=1,\quad 0\leq c^*<+\infty.
	\end{equation}

	In particular, if $c^*=0$, we have
\begin{equation}\label{eq: limisup02}
\mathbb{P}(\limsup_{n\rightarrow+\infty}\frac{\alpha'(n)}{f(n)}=0)=1.
\end{equation}

	If $0<c^*\leq+\infty$, let $c_m\uparrow c^*$ as $m\rightarrow+\infty$.
	Then
	$$\mathbb{P}(\limsup_{n\rightarrow+\infty}\frac{\alpha'(n)}{f(n)}\geq c^*)=\mathbb{P}(\lim_{m\rightarrow+\infty}\{\limsup_{n\rightarrow+\infty}\frac{\alpha'(n)}{f(n)}\geq c_m\})=\lim_{m\rightarrow+\infty}\mathbb{P}(\limsup_{n\rightarrow+\infty}\frac{\alpha'(n)}{f(n)}\geq c_m)=1.$$
	Hence, by \eqref{eq: limsup01}, we know that
\begin{equation}\label{eq: limsup03}
\mathbb{P}(\limsup_{n\rightarrow+\infty}\frac{\alpha'(n)}{f(n)}=c^*)=1, \quad 0<c^*\leq +\infty.
\end{equation}

	Therefore, by \eqref{eq: alpha n and alpha'n}, \eqref{eq: limisup02} and \eqref{eq: limsup03} we have that
	$$\mathbb{P}(\limsup_{n\rightarrow+\infty}\frac{\alpha(n)}{f(n)}=c^*)=
	 \mathbb{P}(\limsup_{n\rightarrow+\infty}\frac{\alpha'(n)}{f(n)}=c^*)=1.
	$$
\end{proof}

\begin{proof}[\bf Proof of Lemma \ref{lem: stopping time}.]
For any $q>2$, take $r>\frac{4}{q-2}+2>2$. Then
\begin{equation}\label{eq: r and q}
  \frac{2r}{r-2}=2+\frac{4}{r-2}<q.
\end{equation}
Let $M_n^+:=\max_{0\leq k\leq n}S_k$.
Since $\sum^{\infty}_{n=1}n^{-(1+r)}<+\infty$, by
\cite[Definition 4  on p. 34 and Theorem of Hirsch on p. 39]{Re90},
we get that  for almost every $\omega\in \Omega$, there exists $N(\omega)>0$
such that for all $n>N(\omega)$,
\[n^{-\frac{1}{2}}M_n^+(\omega)>n^{-\frac{1}{r}},\]
i.e.
\begin{equation}\label{eq: Mn and r}
  M_n^{+}(\omega)>n^{\frac{1}{2}-\frac{1}{r}}>1.
\end{equation}
For all $n>N^{\frac{1}{2}-\frac{1}{r}}(\omega)$, we can get $\sigma_n(\omega)<n^{\frac{2r}{r-2}}$.
Indeed, if $\sigma_n(\omega)\geq n^{\frac{2r}{r-2}}$, then
\[\sigma_n(\omega)\geq n^{\frac{2r}{r-2}}>\left(N^{\frac{1}{2}-\frac{1}{r}}(\omega)\right)^{\frac{2r}{r-2}}=N(\omega).\]
Thus by \eqref{eq: Mn and r} we get
\[M^+_{\sigma_n}(\omega)>\sigma_n^{\frac{1}{2}-\frac{1}{r}}\geq \left(n^{\frac{2r}{r-2}}\right)^{\frac{1}{2}-\frac{1}{r}}=n,\]
which contradicts to $M^+_{\sigma_n}(\omega)=n$. Hence $\sigma_n(\omega)<n^{\frac{2r}{r-2}}$. Therefore, by \eqref{eq: r and q} we have
\[0\leq \limsup_{n\to\infty}\frac{\sigma_n(\omega)}{n^q}\leq \limsup_{n\to\infty}\frac{n^{\frac{2r}{r-2}}}{n^q}\to 0,\]
which means
\[\mathbb{P}(\lim_{n\to\infty}\frac{\sigma_n}{n^q}=0)=1.\]
\end{proof}
}

{ \noindent {\bf\large Acknowledgments}\quad
We thank the referees for helpful comments and suggesting, particularly for the suggestion which leads to the current  improved  Theorem \ref{thm: limitation of alpha(n)}. We also hank Yinshan Chang and Longmin Wang for their comments and discussions.
This work was supported by the National Natural Science Foundation of China (12171335, 12001389, 12471139, 12071011), the Simons Foundation (\#960480) and the Science Development Project of Sichuan University (2020SCUNL201).\\

{ \noindent {\bf\large  Declarations}\quad

{\bf Author Contributions:} All authors contribute equally.

{\bf Data Availability:} No datasets were generated or analysed during the current study.

{\bf Competing interests:}  The authors declare no competing interests.

{\bf Ethical approval:} Not applicable.

{\bf Clinical trial number:} Not applicable.


\begin{thebibliography}{10}


\bibitem{B22} Bass, R. F. (2023). The rate of escape of the most visited site of Brownian motion, {\it Electron. J. Probab.} {\bf 28}, Paper No. 20, 12 pp.

\bibitem{BG85} Bass, R. F. and Griffin, P. S. (1985). The most visited site of Brownian motion and simple random walk. {\it Z. Wahrsch. Verw. Gebiete} {\bf 70}, 417--436.

\bibitem{DS18} Ding, J. and Shen, J. (2018). Three favorite sites occurs infinitely often for one-dimensional simple random walk. {\it Ann. Probab.} {\bf 46}, 2545--2561.

\bibitem{Durrett} Durrett, R.  (2019). Probability: theory and examples (5th edition). Cambridge University Press, Cambridge.

\bibitem{ER84} Erd\H{o}s, P. and R\'{e}v\'{e}sz, P. (1984). On the favourite points of a random walk. In {\it Mathematical Structure--Computational Mathematics--Mathematical Modelling} {\bf 2}, 152--157.

\bibitem{ER87} Erd\H{o}s, P. and R\'{e}v\'{e}sz, P. (1987). Problems and results on random walks. In {\it Mathematical Statistics and Probability Theory, Vol. B (Bad Tatzmannsdorf, 1986)} 59--65. Reidel, Dordrecht.

\bibitem{ER91} Erd\H{o}s, P. and R\'{e}v\'{e}sz, P. (1991). Three problems on the random walk in $\mathbf{Z}^d$. {\it Studia Sci. Math. Hungar.} {\bf 26}, 309--320.

\bibitem{Feller68} Feller, W. (1968). An Introduction to Probability Theory and Its Applications Volume 1 (Third Edition), John Wiley \& Sons, Inc.

\bibitem{FH25} Feng, C. X., Hao, C. X. (2025). The exact limsup constant for once-visited sites of one-dimensional simple random walk, arXiv: 2511.21602v1.

\bibitem{H23-b} Hao, C. X. (2023). The escape rate of favorite edges of simple random walk, arXiv: 2303.13210v1.


\bibitem{HHMS22} Hao, C. X., Hu, Z. C., Ma, T. and Song, R. (2024). Three favorite edges occurs infinitely often for one-dimensional simple random walk.
{\it Commun. Math. Stat.} (published online).

\bibitem{HHMS23} Hao, C. X., Hu, Z. C., Ma, T. and Song, R. (2023). Favorite downcrossing sites for one-dimensional simple random walk (in Chinese), {\it J. Sichuan Univ. (Nat. Sci. Ed.)} {\bf 60}, 051002.

\bibitem{Ma88} Major, P. (1988).  On the set visited once by a random walk.
{\it Probab. Th. Rel. Fields} {\bf 77}, 117--128.


\bibitem{Ne84} Newman, D. (1984). In a random walk the number of ``unique experiences" is two on the average, {\it SIAM Review} {\bf 26}, 573--574.

\bibitem{Re90} R\'{e}v\'{e}sz, P. (1990). Random Walk in Random and Non-Random Environment. World Scientific, Singapore.

\bibitem{ST00} Shi, Z. and T\'{o}th, B. (2000). Favourite sites of simple random walk. {\it Period. Math. Hungar.} {\bf 41}, 237--249.


\bibitem{To96} T\'{o}th, B. (1996). Multiple covering of the range of a random walk on $\mathbf{Z}$ (on a question of P. Erd\"{o}s and P. R\'{e}v\'{e}sz), {\it Studia Sci. Math. Hungar.} {\bf 31}, 355--359.


\bibitem{To01} T\'{o}th, B. (2001). No more than three favorite sites for simple random walk. {\it Ann. Probab.} {\bf 29}, 484--503.


\bibitem{TW97} T\'{o}th, B. and Werner, W. (1997). Tied favourite edges for simple random walk. {\it Combin. Probab. Comput.} {\bf  6}, 359--369.

\end{thebibliography}
\end{document}